\documentclass[11pt]{article}
\usepackage[a4paper, total={6in, 8in}]{geometry}
\usepackage{enumerate}
\usepackage{amsmath}
\usepackage{amssymb,latexsym}
\usepackage{amsthm}
\usepackage{color}
\usepackage{cancel}
\usepackage{graphicx}
\usepackage{mathtools}

\usepackage{hyperref}
\usepackage{comment}
\usepackage{amscd}
\usepackage{cite}
\usepackage{lineno}

\makeatletter
\usepackage{comment}
\let\wfs@comment@comment\comment
\let\comment\@undefined

\usepackage{changes}
\let\wfs@changes@comment\comment
\let\comment\@undefined

\newcommand\comment{%
    \ifthenelse{\equal{\@currenvir}{comment}}
    {\wfs@comment@comment}
    {\wfs@changes@comment}%
}

\definechangesauthor[name=Daniele, color=red]{DAN}
\definechangesauthor[name=Giovanni, color=blue]{GIO}
\definechangesauthor[name=Nico, color=green]{NICO}

\newtheorem{theorem}{Theorem}[section]

\newtheorem{lemma}[theorem]{Lemma}

\newtheorem{definition}[theorem]{Definition}
\newtheorem{proposition}[theorem]{Proposition}

\newtheorem{remark}[theorem]{Remark}

\newtheorem*{theorem*}{Main Theorem}

\newcommand\PG{\mathrm{PG}}
\newcommand\Q{\mathrm{Q}}
\newcommand\mO{\mathrm{O}}
\newcommand\fq{\mathbb{F}_q}

\newcommand\fqn{\mathbb{F}_{q^n}}

\title{On the classification of low-degree ovoids of $Q^+(5,q)$}
\author{Daniele Bartoli\thanks{Department of Mathematics and Informatics, University of Perugia, Perugia, Italy.
Email address: daniele.bartoli@unipg.it}\hspace{0.15 cm},  Nicola Durante\thanks{Department of Mathematics and applications R. Caccioppoli, University of Naples Federico II, Naples, Italy.
Email address: ndurante@unina.it}, Giovanni Giuseppe Grimaldi\thanks{Department of Mathematics and applications R. Caccioppoli, University of Naples Federico II, Naples, Italy.
Email address: giovannigiuseppe.grimaldi@unina.it} }
\date{}
\begin{document}

\maketitle

\begin{abstract}
Ovoids of the Klein quadric $Q^+(5,q)$ of $\PG(5,q)$ have been studied in the last 40 year, also because of their connection
with spreads of $\PG(3,q)$ and hence translation planes.
Beside the classical example given by a three dimensional  elliptic quadric (corresponding to the regular spread of $\PG(3,q)$) many other classes of examples are known. First of all the other examples  (beside the elliptic quadric)  of ovoids of $Q(4,q)$ give also examples of ovoids of  $Q^+(5,q)$. Another important class of ovoids of $Q^+(5,q)$ is given by the ones associated to a flock of a three dimensional quadratic cone. To every ovoid of $Q^+(5,q)$ two bivariate polynomials
$f_1(x,y)$ and $f_2(x,y)$ can be associated. 
In this paper, we classify ovoids of $Q^+(5,q)$ such that $f_1(x,y)=y+g(x)$ and $\max\{\deg(f_1),\deg(f_2)\}<(\frac{1}{6.3}q)^{\frac{3}{13}}-1$, that is $f_1(x,y)$ and $f_2(x,y)$ have "low degree" compared with $q$. 
\end{abstract}
\section{Introduction}\label{Sec:Intro}

Let  $\fqn$, $q=p^{\ell}$, be the finite field with $q^n$ elements. Denote by  $tr(x)=x+x^p+\cdots+ x^{p^{\ell n-1}}$ the {\em absolute trace} of an element $x$ of $\fqn$.

 Denote by $\mathbb{P}$ a {\em finite classical polar space}, i.e. the set of absolute points of either a polarity or a non-singular quadratic form of a projective space $\PG(m,q)$.
The maximal dimensional projective subspaces contained in $\mathbb{P}$ are the {\em generators} of $\mathbb{P}$.

\begin{definition}\label{Definition}
Let $\mathbb{P}$ be a finite classical polar space of $\PG(m,q)$. An {\em ovoid} of $\mathbb{P}$ is a set of points of $\mathbb{P}$ that has exactly one point in common with each generator of $\mathbb{P}.$
An ovoid ${\mO}$ of $\mathbb{P}$ is a {\em translation} ovoid with respect to a point $P {\in}{ \mO}$ if there is a collineation group of $\mathbb{P}$  fixing $P$ linewise (i.e. stabilizing all lines through $P$) and acting sharply transitively on the points of ${\mO} \setminus \{P\}$.
\end{definition}

Regarding translation ovoids of orthogonal polar spaces the following is known.

\begin{theorem}\cite{LunPol}
Translation ovoids of an orthogonal polar space $\mathbb{P}$ of $\PG(m,q)$  exist if and only if $\mathbb{P}$  is either the hyperbolic quadric $\Q^+(3, q)$ or the hyperbolic quadric $\Q^+(5, q)$ or the parabolic quadric $Q(4,q).$
\end{theorem}

\begin{itemize}
\item Examples of translation ovoids of the  hyperbolic quadric $\Q^+(3,q)$ are the conics contained in it  or  the image of a pseudoregulus  of $\PG(3,q)$ (see \cite{Fre,LunMarPolTro}) under the Klein correspondence  (see \cite{Dur19}). 
\item  The ovoids of the Klein quadric $\Q^+(5,q)$ correspond to line spreads of $\PG(3, q)$ and translation ovoids are equivalent to semifield spreads. Hence $\Q^+(5, q)$ has ovoids and translation ovoids for all values of $q$.
\item If $\Q(4, q) = H\cap \Q^+(5, q)$ is a non-degenerate quadric, where $H$ is a non-tangent hyperplane of $\PG(5, q)$, then ovoids of $\Q (4, q)$ are equivalent to symplectic spreads of $\PG(3, q)$ and translation ovoids are equivalent to symplectic semifield spreads of $\PG(3, q)$.
\end{itemize}

\noindent For more results on ovoids of  finite classical polar spaces see e.g. \cite{DebKleMet}.


\noindent In this paper we will focus on ovoids of $Q^+(5,q).$  {Denote by $(X_0,X_1,X_2,X_3,X_4,X_5)$ the homogeneous coordinates of $\PG(5,q)$ and by $H_\infty$  the hyperplane at infinity  $X_0=0$.} 

\noindent Up to collineations, in $\PG(5,q)$ there are  two non-degenerate quadrics, the elliptic quadric $Q^-(5,q)$
and the hyperbolic quadric $Q^+(5,q)$, also called the Klein quadric. In the sequel we fix the following hyperbolic quadric:

\[ \Q:  X_0X_5+X_1X_4+X_2X_3=0.\]

By Definition \ref{Definition}, an ovoid of $Q^+(5,q)$ is a set of $q^2+1$ pairwise non-collinear  points  on the quadric.  We can always assume that an ovoid $\mO_5$ of $Q^+(5,q)$ contains the points $(1,0,0,0,0,0)$ and $(0,0,0,0,0,1)$ and hence it  can be written in the following form:
\begin{equation}\label{Eq:Param}
\mO_5(f_1,f_2) = \{(1,x,y,f_1(x,y),f_2(x,y),-xf_2(x,y)-yf_1(x,y))\}_{x,y\in \fq} \cup \{(0,0,0,0,0,1)\}, \end{equation}

\noindent for some function $f_i:\fq^2 \longrightarrow \fq$ with $f_i(0,0) =0$, $i\in \{1,2\}$ (see \cite{Williams}).

\noindent  {The set $\mO_5(f_1,f_2)$ is an ovoid if and only if }
\begin{eqnarray}\label{Eq:Intro2} \langle (1,x_1,y_1,f_1(x_1,y_1),f_2(x_1,y_1),-x_1f_2(x_1,y_1)-y_1f_1(x_1,y_1)),\\ (1,x_2,y_2,f_1(x_2,y_2),
f_2(x_2,y_2),-x_2f_2(x_2,y_2)-y_2f_1(x_2,y_2))\rangle \ne 0,\nonumber
\end{eqnarray}
 
\noindent for every $(x_1,y_1) \ne (x_2,y_2)$ in $\mathbb{F}_q^2$, where $\langle  \cdot,\cdot \rangle$ is the symmetric for $q$ odd (or alternating for $q$ even) form associated to the quadratic form of the quadric $Q^+(5,q)$.
Note that  
$$\langle (1,x_1,y_1,f_1(x_1,y_1),f_2(x_1,y_1),-x_1f_2(x_1,y_1)-y_1f_1(x_1,y_1), (0,0,0,0,0,1)\rangle =1.$$

In this paper we will put more attention to ovoids of $Q^+(5,q)$ associated to flocks of a quadratic cone
of $\PG(3,q)$ (see \cite{Tha, Wal}).
These ovoids are union of $q$ conics sharing a point and this property characterizes the ovoids of $Q^+(5,q)$ associated to flock of a quadratic cone of $\PG(3,q)$.
Indeed the following holds:

\begin{theorem} (\cite{GevJoh}). For each ovoid $O$ of $Q^+(5, q)$, which is the
union of $q$ conics sharing a point, there is a flock ${\cal F}$of  a three-dimensional quadratic cone  such that $O$ is
associated to ${\cal F}$.
\end{theorem}

\noindent The known ovoids of $Q^+(5,q)$ contained in a secant  hyperplane section or arising from a flock of the three-dimensional quadratic cone are listed in  Table \ref{Table1}, where  $n$ is a non-square of $\fq$, $q=p^\ell$ odd,   $\sigma\ne 1 $ is an automorphism of $\fq$,
$a\ne 1$ is an element with $tr(a)=1$ of $\fq$, $q$ even, and $n_1, n_2$ are non-squares of $\fq$, $q=3^\ell$; see \cite{GevJoh}, \cite{HandFTP}, \cite{Tha}.

\begin{center}
\tabcolsep= 1 mm
\begin{table}[ht]
\caption{Known ovoids of $Q^+(5,q)$.}\label{Table1}
\vspace{0.25cm}
\begin{center}
 \begin{tabular}{||c c c c||} 
\hline
Name  &  $f_1(x,y)$ & $f_2(x,y)$ &  Restrictions  \\ [0.5ex] 
 \hline\hline
{\em Elliptic quadric} & $y$ & $-nx$ & $q$ odd  \\ 
 \hline
{\em Elliptic quadric} &  $y$ & $ax+y$ & $q$ even  \\ 
 \hline
{\em Kantor } &  $y$ & $-nx^\sigma$ &  $q$ odd, $\ell>1$ \\
 \hline
{\em Penttila-Williams}  &  $y$ & $-x^9-y^{81}$ & $p=3$, $\ell=5$   \\
 \hline
{\em Thas-Payne}  &  $y$ & $-nx-(n^{-1} x)^{1/9} - y^{1/3}$ & $p=3, \ell>2$   \\ 
\hline
{\em Ree-Tits slice}  &  $y$ & $-x^{2\sigma+3}-y^\sigma$ & \begin{tabular}{c}$p=3, \ell>1$,\\ $\ell$ odd, $\sigma=\sqrt{3q} $\end{tabular}\\
 \hline
{\em Tits} &  $y$ & $x^{\sigma+1}+ y^\sigma$ & \begin{tabular}{c}$p=2, \ell>1$ \\$\ell$ odd, $\sigma = \sqrt{2q}$\end{tabular}  \\ 
\hline
{\em Fisher-Thas-Walker} & $y-x^2$ & $x^3/3$ & $q \equiv -1$ mod $3$ \\ 
\hline
{\em Kantor-Payne} & $y-\beta x^3$ & $\gamma x^5$ & \begin{tabular}{c}$p \equiv \pm 2$ mod $5$\\ $\ell$ odd, $\beta^2=5\gamma$\end{tabular}\\   
\hline
{\em Law-Penttila} & $y-x^4-nx^2$ & \begin{tabular}{c}$-n^{-1}x^9+x^7$\\ $+n^2x^3-n^3x$\end{tabular} & $p=3$\\  
\hline
{\em Ganley} & $y-dx^3$ & $-n_1x^9-n_1n_2^2x$ & $p=3$, $\ell > 2$, $d^2 = n_1n_2$ \\  
\hline
{\em Thas} & $y-bx$ & $cx$ & \begin{tabular}{c}$t^2+bt+c$ is\\ irreducible over $\mathbb{F}_q$\end{tabular} \\   
\hline
{\em Kantor} & $y-x^2$ & $\frac{1}{3}x^3-nx^5-n^{-1}x$ & $p=5$\\
\hline
{\em Penttila-Williams} & $y+x^{27}$ & $-x^9$ & $p=3$, $\ell=5$ \\  
\hline
\end{tabular}
\end{center}
\end{table}
\end{center}

Note that  Condition \eqref{Eq:Intro2} reads

\begin{equation}\label{Eq:intro}
(x_1-x_2)(f_2(x_2,y_2)-f_2(x_1,y_1)) + (y_1-y_2) (f_1(x_2,y_2)-f_1(x_1,y_1)) \ne 0,\end{equation}
for every   $(x_1,y_1) \ne (x_2,y_2)$ in $\mathbb{F}_q^2$.

{In Section \ref{Sec:AlgebraicVarieties}, using \eqref{Eq:intro}, a hypersurface $\mathcal{S}_{f_1,f_2}$ of $\PG(5,q)$ will be attached to $\mO_5(f_1,f_2)$. 
Investigating the existence of absolutely irreducible components defined over $\mathbb{F}_q$ in  $\mathcal{S}_{f_1,f_2}$, we will provide the following classification result, which is also the main achievement of our paper.

\begin{theorem*}\label{mainthm}
Let $f_1(X,Y)=Y+\sum_ia_iX^i$ and $f_2(X,Y)=\sum_{ij}b_{i,j}X^iY^j$. Suppose that $q>6.3 (\max\{d_1,d_2\}+1)^{13/3}$, where $d_1=\deg(f_1), d_2=\deg(f_2)$. Let $\Gamma:=\{i:b_{i,0}\neq 0, 0<i<d_2 \}$ and denote 
$$h:= \begin{cases}
\max (\Gamma)& \textrm{ if } \Gamma\neq \emptyset;\\ 
\infty & \textrm{ if } \Gamma= \emptyset.\\ 
\end{cases}
$$
If $\mO_5(f_1,f_2)$ is an ovoid of $\Q^+(5,q)$ then either  $d_2=d_1=1$ or 
$b_{i,j}=0$ for any  $j\neq 0$ and one of the following holds:
\begin{enumerate}
    \item $p=2$, $1<d_1<d_2\leq 2d_1-1$, $d_2$ odd;
    \item $p=3$, $h+1<2d_1< d_2=3^j$ and $3 \mid d_1$;
    \item $p=3$, $h+1=2d_1<d_2=3^j$ and $3 \mid (d_1-1)$;
    \item $p\geq 3$ and  $d_2=2d_1-1\leq 11$;
    \item $p\geq 3$, $h=\infty$, $d_1=1$, $d_2=p^j$,  $a_{d_1}=0$, and $O_5(Y,f_2)$ corresponds to an ovoid  $O_4(f_2)$ of $Q(4,q)$;
    \item $p>3$ and $(p,d_1,d_2,h)\in\{(5,2,5,3),(7,2,7,3),(7,3,7,5)\}$.

\end{enumerate}

\end{theorem*}

\begin{remark}
Note that when $f_1(X,Y)=Y$ (i.e. case 5. in the theorem above), the classification of $O_5(Y,f_2)$ follows directly from the classification of  $O_4(f_2)$ of $Q(4,q)$ which has been achieved in \cite[Main Theorem]{DanNico2021}. Therefore the unique examples in small degree regime are the elliptic quandric and the Kantor ovoid.

The other low degree examples in Table \ref{Table1} fall in Case 1. (Fisher-Thas-Walker,Kantor-Payne), Case 2. (Ganley), Case 3. (Law-Penttila),  Case 4. (Fisher-Thas-Walker, Kantor-Payne), Case 6. (Kantor).
\end{remark}


\section{Link between ovoids of \texorpdfstring{$Q^+(5,q)$}{Lg} and algebraic varieties}\label{Sec:AlgebraicVarieties}
We refer the reader to \cite{BartoliSurvey} for a survey on links between algebraic varieties over finite fields and relevant combinatorial objects. 

An algebraic hypersurface $\mathcal{S}$ is an algebraic variety that may be defined by a single implicit equation. An algebraic hypersurface defined over a field $\mathbb{K}$  is \emph{absolutely irreducible}  if the associated polynomial is irreducible over every algebraic extension of $\mathbb{K}$. An absolutely irreducible $\mathbb{K}$-rational component of a hypersurface $\mathcal{S}$, defined by the polynomial $F$, is simply an absolutely irreducible hypersurface such that the associated polynomial has coefficients in $\mathbb{K}$ and it is a factor of $F$.   We will make use of the homogeneous equations for algebraic varieties. For a deeper introduction to algebraic varieties we refer the interested reader to \cite{Hartshorne}.

As mentioned in Section \ref{Sec:Intro} an ovoid of $Q^+(5,q):  X_0X_5+X_1X_4+X_2X_3=0$ can be written in the  form 
$$\mO_5(f_1,f_2) = \{(1,x,y,f_1(x,y),f_2(x,y),-xf_2(x,y)-yf_1(x,y)\}_{x,y\in \fq} \cup \{(0,0,0,0,0,1)\},
$$
for some function $f_i$ with $f_i(0,0) =0$, $i\in \{1,2\}$ and 
$$(x_1-x_2)(f_2(x_2,y_2)-f_2(x_1,y_1))+(y_1-y_2)(f_1(x_2,y_2)-f_1(x_1,y_1)) \ne 0
$$
for every   $(x_1,y_1) \ne (x_2,y_2)$ in $\mathbb{F}_q^2$. Let $\widetilde f_i(X,Y,T)$ be the homogenization of $f_i(X,Y)$ and let $d_i=\deg(f_i)$. In order to obtain non-existence results and partial classifications for the ovoids of $Q^+(5,q)$, we will consider the hypersurface  $\mathcal{S}_{f_1,f_2}\subset \mathrm{PG}(4,q)$ defined by $F(X_0,X_1,X_2,X_3,X_4)=0$, where 
\begin{eqnarray}
F(X_0,X_1,X_2,X_3,X_4) &=&  (X_1-X_3) \Big(\widetilde f_2(X_3,X_4,X_0)-\widetilde f_2(X_1,X_2,X_0)\Big) \nonumber\\
&&+(X_2-X_4)\Big(\widetilde f_1(X_3,X_4,X_0)-\widetilde f_1(X_1,X_2,X_0)\Big). \label{Eq:F}
\end{eqnarray}

\begin{lemma}\label{lemma}
The set $\mO_5(f_1,f_2)$ is an ovoid of $Q^+(5,q)$ if and only if $\mathcal{S}_{f_1,f_2}$ contains no affine $\mathbb{F}_q$-rational points off the three-dimensional subspace $X_1-X_3=0=X_2-X_4$. 
\end{lemma}

A fundamental tool to determine the existence of rational points in an algebraic variety over finite fields is the following result by Lang and Weil \cite{LangWeil} in 1954, which can be seen as a generalization of the Hasse-Weil bound.
\begin{theorem}\label{Th:LangWeil}[Lang-Weil Theorem]
Let $\mathcal{V}\subset \mathbb{P}^N(\mathbb{F}_q)$ be an absolutely irreducible variety of dimension $n$ and degree $d$. Then there exists a constant $C$ depending only on $N$, $n$, and $d$ such that 
\begin{equation}\label{Eq:LW}
\left|\#\mathcal{V}(\mathbb{F}_q)-\sum_{i=0}^{n} q^i\right|\leq (d-1)(d-2)q^{n-1/2}+Cq^{n-1}.
\end{equation}
\end{theorem}

\noindent Although the constant $C$ was not  computed in \cite{LangWeil}, explicit estimates have been provided for instance in  \cite{CafureMatera,Ghorpade_Lachaud,Ghorpade_Lachaud2,LN1983,WSchmidt,Bombieri} and they have the general shape $C=f(d)$ provided that $q>g(n,d)$, where $f$ and $g$ are polynomials of (usually) small degree. We refer to \cite{CafureMatera} for a survey on these bounds. Excellent surveys on Hasse-Weil and Lang-Weil type theorems are \cite{Geer,Ghorpade_Lachaud2}. We include here the following result by Cafure and Matera \cite{CafureMatera}. 

\begin{theorem}\cite[Theorem 7.1]{CafureMatera}\label{Th:CafureMatera}
Let $\mathcal{V}\subset\mathrm{AG}(n,q)$ be an absolutely irreducible $\mathbb{F}_q$-variety of dimension $r>0$ and degree $\delta$. If $q>2(r+1)\delta^2$, then the following estimate holds:
$$|\#(\mathcal{V}\cap \mathrm{AG}(n,q))-q^r|\leq (\delta-1)(\delta-2)q^{r-1/2}+5\delta^{13/3} q^{r-1}.$$
\end{theorem}

For our purposes, the existence of an absolutely irreducible $\mathbb{F}_q$-rational component in $\mathcal {S}_{f_1,f_2}$ is enough to provide asymptotic non-existence results.

\begin{theorem}\label{Th:Main}
Let $\mathcal{S}_{f_1,f_2}: F(X_0,X_1,X_2,X_3,X_4)=0$, where $F$ is defined as in \eqref{Eq:F}. Suppose that $q>6.3 (d+1)^{13/3}$, $\max\{\deg(f_1),\deg(f_2)\}=d$, and $\mathcal{S}_{f_1,f_2}$ contains an absolutely irreducible component $ \mathcal{V}$ defined over $\mathbb{F}_q$. Then  $\mO_5(f_1,f_2)$ is not an ovoid of $Q^+(5,q)$.
\end{theorem}
\proof
 Since $q>6.3 (d+1)^{13/3}$, by Theorem \ref{Th:CafureMatera}, with $\delta=d+1$ and $r=3$, it is readily seen that $\mathcal{V}$ contains more than $q^2$ affine $\mathbb{F}_q$-rational points. This yields the existence of at least one affine $\mathbb{F}_q$-rational point $(a,b,c,d)$ in $\mathcal{V}\subset \mathcal{S}_{f_1,f_2}$ such that $a\neq c$ or $b\neq d$ and therefore $\mO_5(f_1,f_2)$ is not an ovoid of $Q^+(5,q)$.
\endproof

We also include, for seek of completeness, the following results  which will be useful  in the sequel. 

\begin{lemma}\cite[Lemma 2.1]{Aubry}\label{Lemma:Aubry} Let $\mathcal{H},\mathcal{S}\subset \mathrm{PG}(n,q)$ be two projective hypersurfaces. If $\mathcal{H} \cap  \mathcal{S}$ has a reduced absolutely irreducible component defined over $\mathbb{F}_q$ then $\mathcal{S}$ has an absolutely irreducible component defined over $\mathbb{F}_q$.
\end{lemma}

\begin{lemma}\cite[Lemma 2.9]{DanYue2020}\label{Lemma:DanYue2020} Let $\mathcal{S}$ be a hypersurface containing $O=(0, 0,\ldots, 0)$ of  affine equation $F(X_1, ..., X_n) =0$, where 
$$
F(X_1,\ldots,X_n)=F_d(X_1,\ldots,X_n)+F_{d+1}(X_1,\ldots,X_n)+\cdots,
$$
with $F_i$ the homogeneous part of degree  $i$ of $F(X_1,\ldots,X_n)$ for $i =d, d +1,\dots$. Let $P$ be an $\mathbb{F}_q$-rational simple point of the variety $F_d(X_1,\ldots,X_n)=0$. Then there exists an $\mathbb{F}_q$-rational plane $\pi$ through the line $\ell$ joining $O$ and $P$ such that $\pi \cap \mathcal{S}$ has $\ell$ as a non-repeated tangent $\mathbb{F}_q$-rational line at the origin and $\pi \cap \mathcal{S}$ has a non-repeated absolutely irreducible $\mathbb{F}_q$-rational component.
\end{lemma}


\section{General results}

In view of Theorem \ref{Th:Main}, our goal is to find conditions for which $\mathcal{S}_{f_1,f_2}$  defined by

$$(X_2-X_4)\left[\sum_{i,j} a_{ij}
\left(X_3^i X_4 ^j -X_1^iX_2^j\right)X_0^{d-i-j}\right] + (X_1-X_3) \left[\sum_{i,j} b_{ij}
\left(X_3^i X_4 ^j -X_1^iX_2^j\right)X_0^{d-i-j}\right]$$ possesses an absolutely irreducible component defined over $\mathbb{F}_q$.  Recall that $f_1(0,0)=0$ and $f_2(0,0)=0$, therefore  $a_{0,0}=0$ and $b_{0,0}=0$. 

\begin{remark}\label{remark:b_01}
Note that we can always assume that one between $a_{1,0}$ and $b_{0,1}$ is zero. In the following we will always assume $b_{0,1}=0$.
\end{remark}

\begin{proposition}\label{d1d2}

The hypersurface $\mathcal{S}_{f_1,f_2}$ contains an absolutely irreducible $\mathbb{F}_q$-rational component in the following cases:
\begin{enumerate}
    \item $d_1=d_2$, $a_{0,d}b_{d,0}=0$.

   \item $d=d_1>d_2$ and  there exists $i\neq 0$ such that $a_{i,d-i}\neq 0$.

\item $d=d_2>d_1$ and  there exists $i\neq d$ such that $b_{i,d-i}\neq 0$.
\end{enumerate}
\end{proposition}
\begin{proof}
Consider the  $\mathcal{S}^{\prime}:=\mathcal{S}_{f_1,f_2}\cap (X_0=0)$. 
\begin{enumerate}
    \item Suppose $d_1=d_2=d$. Consider $\mathcal{S}^{\prime\prime}=\mathcal{S}^{\prime}\cap (X_3=0)$ given by 

\begin{eqnarray*}
-(X_2 - X_4)\left[\sum_{i} a_{i,d-i}
 X_1^iX_2^{d-i}\right] -X_1 \left[\sum_{i} b_{i,d-i}
 X_1^iX_2^{d-i}\right]\\+a_{0,d}(X_2-X_4)X_4^d+b_{0,d}X_1X_4^d&=&0,
\end{eqnarray*}

which is of degree $1$ in $X_4$ if $a_{0,d}=b_{0,d}=0$.

If $a_{0,d}=0$ and $b_{0,d}\neq 0$ then $X_1$ is a non-repeated factor of $\mathcal{S}^{\prime\prime}$.

Consider $\mathcal{S}^{\prime\prime}=S^{\prime}\cap (X_2=0)$ given by 

\begin{eqnarray*}
- X_4\left[\sum_{i} a_{i,d-i}
 X_3^iX_4^{d-i}\right] +(X_1-X_3) \left[\sum_{i} b_{i,d-i}
 X_3^iX_4^{d-i}\right]\\ + a_{d,0}X_4 X_1^d-b_{d,0}(X_1-X_3)X_1^d &=& 0,
\end{eqnarray*}

which is of degree $1$ in $X_1$ if $a_{d,0}=b_{d,0}=0$.

If $b_{d,0}=0$ and $a_{d,0}\neq 0$ then $X_4$ is a non-repeated factor of $\mathcal{S}^{\prime\prime}$.

\item $d=d_1>d_2$. The equation of $\mathcal{S}^{\prime}$ reads 
$$(X_2-X_4)\sum_{i\leq d}a_{i,d-i}
\left(X_3^i X_4 ^{d-i} -X_1^iX_2^{d-i}\right)=0.$$

Now, $(X_2-X_4)\nmid G(X_1,X_2,X_3,X_4):=\sum_{i\leq d}a_{i,d-i}
\left(X_3^i X_4 ^{d-i} -X_1^iX_2^{d-i}\right)$ since there exists an element $a_{i,d-i}\neq 0$ with $i\neq 0$ and then $(X_2-X_4)$ is a non-repeated component of $\mathcal{S}^{\prime}$. By Lemma \ref{Lemma:Aubry}, $\mathcal{S}_{f_1,f_2}$ contains an absolutely irreducible component defined over $\mathbb{F}_{q}$ (passing through $H_{\infty}\cap (X_2=X_4))$.

\item $d=d_2>d_1$. Now, the equation of $\mathcal{S}^{\prime}$ reads 
$$(X_1-X_3)\sum_{i\leq d}b_{i,d-i}
\left(X_3^i X_4 ^{d-i} -X_1^iX_2^{d-i}\right)=0.$$

Now, $(X_1-X_3)\nmid G(X_1,X_2,X_3,X_4):=\sum_{i\leq d}b_{i,d-i}
\left(X_3^i X_4 ^{d-i} -X_1^iX_2^{d-i}\right)$ since there exists an element $b_{i,d-i}\neq 0$ with $i\neq d$ and then $(X_1-X_3)$ is a non-repeated component of $\mathcal{S}^{\prime}$. By Lemma \ref{Lemma:Aubry}, $\mathcal{S}_{f_1,f_2}$ contains an absolutely irreducible component defined over $\mathbb{F}_{q}$ (passing through $H_{\infty}\cap (X_1=X_3))$.

\end{enumerate}

\end{proof}

In what follows we denote by $j_1$ the non-negative integer $\max\{j : a_{i,j}\neq 0\}$ and by $j_2$ the non-negative integer $\max\{j : b_{i,j}\neq 0\}$ and  by $i_1$ the non-negative integer $\max\{i : a_{i,j}\neq 0\}$ and by $i_2$ the non-negative integer $\max\{i : b_{i,j}\neq 0\}$.

\begin{proposition}\label{Prop:j_bar}


If $(j_2,j_1)\neq (1,0)$ and $j_2 >j_1$ then $\mathcal{S}_{f_1,f_2}$ contains an absolutely irreducible component defined over $\mathbb{F}_q$.
\end{proposition}
\begin{proof}
The polynomial $F(X_4,X_1,X_2,X_3,X_0)$ reads
$$(X_2-X_0)\left[\sum_{i,j} a_{ij}
\left(X_3^i X_0 ^j -X_1^iX_2^j\right)X_4^{d-i-j}\right] + (X_1-X_3) \left[\sum_{i,j} b_{ij}
\left(X_3^i X_0 ^j -X_1^iX_2^j\right)X_4^{d-i-j}\right],$$and the surface $\widetilde{\mathcal{S}} : F(X_4,X_1,X_2,X_3,X_0)=0$ is projectively equivalent to $\mathcal{S}_{f_1,f_2}$.

We study now the tangent cone $\Omega$ of $\widetilde{\mathcal{S}}$ at the origin, whose equation depends on $j_1$ and $j_2$. We distinguish several cases.

\begin{enumerate}

\item $j_1=0$.
\begin{enumerate}
\item[(i)] $j_2=0$. $\Omega$ reads $\sum_ia_{i,0}(X_3^i-X_1^i)X_4^{d-i}$.

\item[(ii)] $j_2=1$. $\Omega$ reads $(X_1-X_3)\sum_ib_{i,1}X_3^iX_4^{d-i-1} + \sum_ia_{i,0}(X_1^i-X_3^i)X_4^{d-i}$.

\item[(iii)] $j_2 \geq 2$. $\Omega$ reads $(X_1-X_3)\sum_ib_{i,j_2}X_3^iX_4^{d-i-j_2}$. 
\end{enumerate}
\item $j_1 \geq 1$.
\begin{enumerate}
\item[(i)] $j_2 <j_1+1$. $\Omega$ reads $\sum_ia_{i,j_1}X_3^iX_4^{d-i-j_1}$.

\item[(ii)] $j_2=j_1+1$. $\Omega$ reads $\sum_ia_{i,j_1}X_3^iX_4^{d-i-j_1}- (X_1-X_3)\sum_ib_{i,j_2}X_3^iX_4^{d-i-j_2}$.

\item[(iii)] $j_2>j_1+1$. $\Omega$ reads $(X_1-X_3)\sum_ib_{i,j_2}X_3^iX_4^{d-i-j_2}$.
\end{enumerate}
\end{enumerate}

If cases (1iii) and (2iii) hold, then $X_1-X_3=0$ is a non-repeated $\mathbb{F}_q$-rational component of such a tangent cone. Also, there exists a non-singular point $P$ for $\Omega$.   By Lemma \ref{Lemma:DanYue2020} there exists an $\mathbb{F}_q$-rational plane $\pi$ through the line $\ell$ joining the origin  and $P$ such that $\pi\cap \widetilde{\mathcal{S}}$ has $\ell$ as a non-repeated tangent $\mathbb{F}_q$-rational line at the origin and $\pi\cap \widetilde{\mathcal{S}}$ has a non-repeated absolutely irreducible $\mathbb{F}_q$-rational component. Using twice Lemma \ref{Lemma:Aubry}, one deduces the existence of an absolutely irreducible $\mathbb{F}_q$-rational component in $\widetilde{\mathcal{S}}$ and therefore in $\mathcal{S}_{f_1,f_2}$.

 If case (2ii) holds, then $\Omega$ is a rational surface ($X_1$ has degree $1$). This means that $\Omega$ contains an absolutely irreducible $\mathbb{F}_q$-rational component of degree $1$ in $X_1$. It is readily seen that affine  singular points of $\Omega$ satisfy $\sum_i a_{i,j_1}X_3^iX_4^{d-i-j_1}=0$ and therefore there exists at least a   non-singular point $P$ for $\Omega$. The same conclusion as above holds. 

\end{proof}

\begin{proposition}\label{Prop:i_bar}


If $(i_1,i_2)\neq (1,0)$ and $i_1 > i_2$ then $\mathcal{S}_{f_1,f_2}$ contains an absolutely irreducible component defined over $\mathbb{F}_q$.
\end{proposition}
\begin{proof}
The polynomial $F(X_3,X_1,X_2,X_0,X_4)$ reads
$$(X_2-X_4)\left[\sum_{i,j} a_{ij}
\left(X_0^i X_4 ^j -X_1^iX_2^j\right)X_3^{d-i-j}\right] + (X_1-X_0) \left[\sum_{i,j} b_{ij}
\left(X_0^i X_4 ^j -X_1^iX_2^j\right)X_3^{d-i-j}\right],$$
and the surface $\widetilde{\mathcal{S}} : F(X_3,X_1,X_2,X_0,X_4)=0$ is projectively equivalent to $\mathcal{S}_{f_1,f_2}$.

We study now the tangent cone $\Omega$ of $\widetilde{\mathcal{S}}$ at the origin, whose equation depends on $i_1$ and $i_2$. We distinguish several cases.

\begin{enumerate}
    \item $i_2=0$.
    \begin{itemize}
        \item[(i)] $i_1=0$. $\Omega$ reads $\sum_j b_{0,j}(X_4^j-X_2^j)X_3^{d-j}$.
        \item[(ii)] $i_1=1$. $\Omega$ reads $(X_2-X_4)\sum_j a_{1,j}X_4^jX_3^{d-j-1} + \sum_j b_{0,j}(X_2^j-X_4^j)X_3^{d-j}$.
        \item[(iii)] $i_1 \geq 2$. $\Omega$ reads $(X_2-X_4)\sum_j a_{i_1,j}X_4^jX_3^{d-j-i_1}$.
    \end{itemize}
    
    \item $i_2 \geq 1$.
    \begin{itemize}
        \item[(i)] $i_1 < i_2+1$. $\Omega$ reads $\sum_j b_{i_2,j}X_4^jX_3^{d-j-i_2}$.
        \item[(ii)] $i_1=i_2+1$. $\Omega$ reads $(X_2-X_4)\sum_j a_{i_1,j}X_4^jX_3^{d-j-i_1}-\sum_j b_{i_2,j}X_4^jX_3^{d-j-i_2}$.
        \item[(iii)] $i_1 > i_2 +1$. $\Omega$ reads $(X_2-X_4)\sum_j a_{i_1,j}X_4^jX_3^{d-j-i_1}$.
    \end{itemize}
\end{enumerate}

If cases (1iii) and (2iii) hold, then $X_2-X_4=0$ is a non-repeated $\mathbb{F}_q$-rational component of such a tangent cone. Also, there exists a non-singular point $P$ for $\Omega$.   By Lemma \ref{Lemma:DanYue2020} there exists an $\mathbb{F}_q$-rational plane $\pi$ through the line $\ell$ joining the origin  and $P$ such that $\pi\cap \widetilde{\mathcal{S}}$ has $\ell$ as a non-repeated tangent $\mathbb{F}_q$-rational line at the origin and $\pi\cap \widetilde{\mathcal{S}}$ has a non-repeated absolutely irreducible $\mathbb{F}_q$-rational component. Using twice Lemma \ref{Lemma:Aubry}, one deduces the existence of an absolutely irreducible $\mathbb{F}_q$-rational component in $\widetilde{\mathcal{S}}$ and therefore in $\mathcal{S}_{f_1,f_2}$.

If case (2ii) holds, then $\Omega$ is a rational surface ($X_2$ has degree $1$). This means that $\Omega$ contains an absolutely irreducible $\mathbb{F}_q$-rational component of degree $1$ in $X_2$. It is readily seen that affine  singular points of $\Omega$ satisfy $\sum_j b_{i_2,j}X_4^jX_3^{d-j-i_2}=0$ and therefore there exists at least a   non-singular point $P$ for $\Omega$. The same conclusion as above holds.

\end{proof}

\begin{theorem}
Let $f_1(X,Y)=\sum_{ij}a_{i,j}X^iY^j$ and $f_2(X,Y)=\sum_{ij}b_{i,j}X^iY^j$. Suppose that $q>6.3 (\max\{d_1,d_2\}+1)^{13/3}$, where $d_1=\deg(f_1), d_2=\deg(f_2)$. Let 
$$
i_1:= \max\{i : a_{i,j}\neq 0\}, 
i_2:= \max\{i : b_{i,j}\neq 0\}, 
j_1:= \max\{j : a_{i,j}\neq 0\}, 
j_2:= \max\{j : b_{i,j}\neq 0\}.
$$

If $\mO_5(f_1,f_2)$ is an ovoid of $\Q^+(5,q)$ then 
$b_{i,j}=0$ if $j\neq 0$ and one of the following holds:
\begin{enumerate}
\item $d_1=d_2$ and $a_{0,d}b_{d,0}\neq 0$;
\item $d_1>d_2$ and $a_{i,d-i}=0$ for all $i\neq 0$;
\item $d_1<d_2$ and $b_{i,d-i}=0$ for all $i\neq d$;
\item $(j_2,j_1)=(1,0)  $ or $j_2\leq j_1$;
\item $(i_1,i_2)=(1,0)  $ or $i_1\leq i_2$.

\end{enumerate}
\end{theorem}

\section{Specific $f_1$}
In what follows, $f_1(X,Y)= Y + \sum_{i}a_{i}X^i, f_2(X,Y)=\sum_{i,j}b_{i,j}X^iY^j$ with $a_{i},b_{i,j}\in \mathbb{F}_q$ for all $i,j$, $d_1=\max\{1,i : a_i \neq0\}$ and $d_2=\max\{i+j: b_{i,j}\neq0\}$. So,

\begin{equation}\label{Eq:S_f}
\begin{aligned}
\mathcal{S}_{f_1,f_2}&: F(X_0,X_1,X_2,X_3,X_4)  :=\hspace{-0.3 cm} & (X_2 - X_4)^2 X_0^{d-1} - (X_2-X_4)\left[\sum_{i} a_{i}
\left(X_3^i  -X_1^i\right)X_0^{d-i}\right] -\\ &&(X_1-X_3) \left[\sum_{i,j} b_{ij}
\left(X_3^i X_4 ^j -X_1^iX_2^j\right)X_0^{d-i-j}\right]=0. 
\end{aligned}
\end{equation}
Let $\Gamma:=\{i:b_{i,0}\neq 0, 0<i<d_2 \}$ and denote 
$$h:= \begin{cases}
\max (\Gamma)& \textrm{ if } \Gamma\neq \emptyset;\\ 
\infty & \textrm{ if } \Gamma= \emptyset.\\ 
\end{cases}
$$

\begin{remark}\label{remark:4.1}
Let $d_1>1$. If $\mathcal{S}_{f_1,f_2}$ does not contain an absolutely irreducible component defined over $\mathbb{F}_q$, then by Proposition \ref{d1d2}, $d_2\geq d_1$. Also, if $d_2> d_1$ then $j_2 \leq j_1$ by Proposition \ref{Prop:j_bar}. If $d_2=d_1>1$ then $a_{0,d}=0$ and so by Proposition \ref{d1d2}, $\mathcal{S}_{f_1,f_2}$ contains an absolutely irreducible component defined over $\mathbb{F}_q$.
So we have to deal with the following two cases
\begin{itemize}
    \item $d_1<d_2$ and $(j_1,j_2)\in \{(1,0),(1,1)\}$;
    \item $d=d_1=d_2=1$.
\end{itemize}
\end{remark}

\begin{remark}\label{remark:fattorizzazione}
Let 
$$F(X_1,\ldots,X_n,Y)=Y^2+Yf(X_1,\ldots,X_n)+g(X_1,\ldots,X_n)\in \mathbb{F}_{p^\ell}[X_1,\ldots,X_n],$$ $p$ odd. Then $F$ is reducible if and only if there exists $r,s\in \overline{\mathbb{F}_{p^{\ell}}}[X_1,\ldots,X_n]$ and $\epsilon \in \overline{\mathbb{F}_ {p^{\ell}}}^*$ such that $f=r+s$, $g=rs$. In fact, from 
$$Y^2+Yf(X_1,\ldots,X_n)+g(X_1,\ldots,X_n)=\Big( Y+\frac{\alpha_1(X_1,\ldots,X_n)}{\alpha_2(X_1,\ldots,X_n)}\Big)\Big( Y+\frac{\beta_1(X_1,\ldots,X_n)}{\beta_2(X_1,\ldots,X_n)}\Big),$$
with $\alpha_1,\alpha_2,\beta_1,\beta_2\in \overline{\mathbb{F}_{p^\ell}}(X_1,\ldots,X_n)$ and $\gcd(\alpha_1,\alpha_2)=\gcd(\beta_1,\beta_2)=1$, one deduces $f=\alpha_1/\alpha_2+\beta_1/\beta_2$, $g=(\alpha_1\beta_1)/(\alpha_2\beta_2)$. So $\beta_1=\gamma_1\alpha_2$ and $\alpha_1=\gamma_2\beta_2$ for some $\gamma_1,\gamma_2\in \overline{\mathbb{F}_{p^\ell}}[X_1,\ldots,X_n]$. Therefore 
$$f=\frac{\alpha_1\beta_2+\beta_1\alpha_2}{\alpha_2\beta_2}= \frac{\gamma_2\beta_2^2+\gamma_1\alpha_2^2}{\alpha_2\beta_2}$$
 and thus $\beta_2\mid \gamma_1\alpha_2^2$ and  $\alpha_2\mid \gamma_2\beta_2^2$. If $\alpha_2$ or $\beta_2$ were non-constant then they would have common factors with $\gamma_1\alpha_2=\beta_1$ or $\gamma_2\beta_2=\alpha_1$, a contradiction.  The claim follows. 

\end{remark}

\section{$p=2$}



\begin{proposition}\label{Prop_p2_1}
Let $\mathcal{S}_{f_1,f_2}$ be as in \eqref{Eq:S_f}, $d=d_2 > d_1$. If $j_2=1$  then  $\mathcal{S}_{f_1,f_2}$ contains an absolutely irreducible component defined over $\mathbb{F}_q$.
\end{proposition}
\begin{proof}

Let $\mathcal{W}=\mathcal{S}_{f_1,f_2} \cap (X_2=X_4)$ be defined by 
\begin{eqnarray*}
&& (X_1+X_3) \left[\sum_{i,j} b_{ij}X_2^jX_0^{d-i-j}
\left(X_1^i  +X_3^i\right)\right]\\
&=&(X_1+X_3) \left( X_2 \sum_{0< i\leq d-2} b_{i,1}(X_1^i  +X_3^i)X_0^{d-i-1}+ \sum_{0< i \leq d} b_{i,0}(X_1^i  +X_3^i)X_0^{d-i}\right).
\end{eqnarray*}
If some of the $b_{i,1}$'s, $0<i\leq d-2$, is nonzero then  $\mathcal{W}$ contains a rational component of degree $1$ in $X_2$ and therefore $\mathcal{S}_{f_1,f_2}$ contains an absolutely irreducible component defined over $\mathbb{F}_q$ through it. 
\end{proof}

\begin{proposition}\label{Prop:p_even_d1=1}
Let $\mathcal{S}_{f_1,f_2}$ be as in \eqref{Eq:S_f}, $d=d_2 > d_1$.  If $d_1= 1$  then  $\mathcal{S}_{f_1,f_2}$ contains an absolutely irreducible component defined over $\mathbb{F}_q$.
\end{proposition}
\proof

By Proposition \ref{Prop_p2_1}, $j_2=0$ and therefore  $F(X_0,X_1,X_2,X_3,X_4)$ reads
\begin{eqnarray*}
(X_2 +X_4)^2X_0^{d-1} + a_1(X_1+X_3)(X_2+X_4)X_0^{d-1} + (X_1+X_3)\sum_{0<i\leq d}  b_{i,0} \left(X_1^i  +X_3^i\right)X_0^{d-i}.
\end{eqnarray*}

If $a_1 = 0$ then $\mathcal{S}_{f_1,f_2}$ is clearly a rational surface (in $Z=(X_2+X_4)^2$) and therefore it contains an absolutely irreducible component defined over $\mathbb{F}_q$.

Suppose now $a_1 \neq 0$ and consider the affine part of $\mathcal{S}_{f_1,f_2}:F(1,X_1,X_2,X_3,X_4)=0.$ It can be written as 

\begin{equation}\label{Eq:Y}
Z^2+Z+\frac{\sum_{0<i\leq d}  b_{i,0}
\left(X_1^i  +X_3^i\right)}{a_1^2(X_1+X_3)},
\end{equation}

where $Z=(X_2+X_4)/\left(a_1(X_1+X_3)\right).$
In what follows we will prove that the surface $\mathcal{Y}$ defined by the affine equation \eqref{Eq:Y} is absolutely irreducible. 

Note that  the surface $\mathcal{Y}$ is reducible only if there exists  $\overline{Z}(X_1,X_3)=\frac{\overline{Z}_1(X_1,X_3)}{\overline{Z}_2(X_1,X_3)}$, with $\overline{Z}_i(X_1,X_3)\in \overline{\mathbb{F}_q}[X_1,X_3]$, such that $(\overline{Z}(X_1,X_3))^2+\overline{Z}(X_1,X_3)+\frac{\sum_{0<i\leq d}  b_{i,0}
\left(X_1^i  +X_3^i\right)}{a_{1}^2(X_1+X_3)}=0$. This condition reads 
\begin{equation}\label{Eq:Z}
(\overline{Z}_1(X_1,X_3))^2+\overline{Z}_1(X_1,X_3)\overline{Z}_2(X_1,X_3)+(\overline{Z}_2(X_1,X_3))^2\frac{\sum_{0<i\leq d}  b_{i,0}
\left(X_1^i  +X_3^i\right)}{a_{1}^2(X_1+X_3)}=0.
\end{equation}
Now, $\deg(\overline{Z}_1)>\deg(\overline{Z}_2)$, otherwise  $2\deg(\overline{Z}_1)\leq \deg(\overline{Z}_1)+\deg(\overline{Z}_2)<2\deg(\overline{Z}_2)+d-1$ and \eqref{Eq:Z} cannot be satisfied. 

Also, the same argument shows that $2\deg(\overline{Z}_1)=2\deg(\overline{Z}_2)+d-1$ and therefore the highest homogeneous parts in $(\overline{Z}_1(X_1,X_3))^2$ and in $(\overline{Z}_2(X_1,X_3))^2\frac{\sum_{0<i\leq d}  b_{i,0}
\left(X_1^i  +X_3^i\right)}{a_{1}^2(X_1+X_3)}$ must be equal.

Write  $d=2^jk$, with $j\geq0$ and $k$ odd. The above argument yields that, in particular,
$$L_d(X_1,X_3)=\frac{\left(X_1^d  +X_3^d\right)}{(X_1+X_3)}=\frac{\left(X_1^k  +X_3^k\right)^{2^j}}{(X_1+X_3)},$$
the highest homogenous term in $\frac{\sum_{0<i\leq d}  b_{i,0}
\left(X_1^i  +X_3^i\right)}{a_{1}^2(X_1+X_3)}$,
is a square in  $\overline{\mathbb{F}_q}[X_1,X_3]$. 
\begin{itemize}
    \item If $j>0$ then $(X_1+X_3)^{2^j-1}$ is the highest power of $(X_1+X_3)$ dividing $L_d(X_1,X_3)$. Since  $2^j-1$ is odd,    $L_d(X_1,X_3)$ is not a square in $\overline{\mathbb{F}}_q[X_1,X_3]$ and the surface $\mathcal{Y}$ is irreducible. 
\item If  $j=0$, following the same argument as above, if $\mathcal{Y}$ is reducible then $k=1$, a contradiction to $d>1$.
\end{itemize}
This shows that if $d=d_2>d_1=1$ then $\mathcal{Y}$ and therefore $\mathcal{S}_{f_1,f_2}$ is absolutely irreducible or contains an absolutely irreducible component defined over $\mathbb{F}_q$.
\endproof


\begin{proposition}\label{Prop:p_even_j2=0}
Let $\mathcal{S}_{f_1,f_2}$ be as in \eqref{Eq:S_f}, $d=d_2 > d_1>1$, $(j_1,j_2)= (1,0)$.  If $d_2+1 > 2d_1$
then  $\mathcal{S}_{f_1,f_2}$ contains an absolutely irreducible component defined over $\mathbb{F}_q$.
\end{proposition}
\proof

The polynomial  $F(X_0,X_1,X_2,X_3,X_4)$ reads

\begin{eqnarray*}
(X_2 +X_4)^2X_0^{d-1} + (X_2+X_4)\sum_{i \leq d_1}a_i(X_1^i+X_3^i)X_0^{d-i} +(X_1+X_3)\sum_{0<i\leq d}  b_{i,0}\left(X_1^i  +X_3^i\right)X_0^{d-i}.
\end{eqnarray*}

The affine part of $\mathcal{S}_{f_1,f_2}$ can be written as
\begin{equation}\label{Eq:Y'}
Z^2+Z\frac{\sum_{i \leq d_1}a_i(X_1^i+X_3^i)}{a_{d_1}(X_1^{d_1}+X_3^{d_1})}+(X_1+X_3)\frac{\sum_{0<i\leq d}b_{i,0}(X_1^i+X_3^i)}{a_{d_1}^2(X_1^{d_1}+X_3^{d_1})^2},
\end{equation}
where $Z=(X_2+X_4)/\left(a_{d_1}(X_1^{d_1}+X_3^{d_1})\right)$.

We will prove that the surface $\mathcal{Y}$ defined by the affine equation \eqref{Eq:Y'} is absolutely irreducible.

Now, the surface $Z^2+Z\frac{\sum_{i\leq d_1}a_i(X_1^i+X_3^i)}{a_{d_1}(X_1^{d_1}+X_3^{d_1})}+(X_1+X_3)\frac{\sum_{0<i\leq d}b_{i,0}(X_1^i+X_3^i)}{a_{d_1}^2(X_1^{d_1}+X_3^{d_1})^2}$ is reducible only if there exists $\overline{Z}(X_1,X_3)=\frac{\overline{Z}_1(X_1,X_3)}{\overline{Z}_2(X_1,X_3)}$, with $\overline{Z}_i(X_1,X_3)\in \overline{\mathbb{F}_q}[X_1,X_3]$, such that 
\begin{equation}\label{equa0:Y'}
\begin{aligned}
(\overline{Z}_1(X_1,X_3))^2+\overline{Z}_1(X_1,X_3)\overline{Z}_2(X_1,X_3)\frac{\sum_{i\leq d_1}a_i(X_1^i+X_3^i)}{a_{d_1}(X_1^{d_1}+X_3^{d_1})}&+\\+(\overline{Z}_2(X_1,X_3))^2(X_1+X_3)\frac{\sum_{0<i\leq d}b_{i,0}(X_1^i+X_3^i)}{a_{d_1}^2(X_1^{d_1}+X_3^{d_1})^2}&=0.
\end{aligned}
\end{equation}

Then Equation \eqref{equa0:Y'} implies that

\begin{equation}\label{equa:Y'}
\begin{aligned}
(\overline{Z}_1(X_1,X_3))^2 a_{d_1}^2(X_1^{d_1}+X_3^{d_1})^2+\overline{Z}_1(X_1,X_3)\overline{Z}_2(X_1,X_3)a_{d_1}(X_1^{d_1}+X_3^{d_1}) \sum_{i\leq d_1}a_i(X_1^i+X_3^i)&+\\+(\overline{Z}_2(X_1,X_3))^2(X_1+X_3)\sum_{0<i\leq d}b_{i,0}(X_1^i+X_3^i)&=0.
\end{aligned}
\end{equation}

Since $d_2+1 > 2d_1$,  $\deg (\overline{Z}_1) > \deg(\overline{Z}_2)$, otherwise
$2\deg(\overline{Z}_1) +2d_1 \leq \deg(\overline{Z}_1) + \deg(\overline{Z}_2) +2d_1 < 2\deg(\overline{Z}_2)+d_2+1$ and \eqref{equa:Y'} cannot be satisfied.

Also, the same argument shows that $2\deg(\overline{Z}_1) +2d_1 = 2\deg(\overline{Z}_2)+d_2+1$ and therefore the highest homogeneous parts in $(\overline{Z}_1(X_1,X_3))^2a_{d_1}^2(X_1^{d_1}+X_3^{d_1})^2$ and in $(\overline{Z}_2(X_1,X_3))^2(X_1+X_3)\sum_{0<i\leq d}b_{i,0}(X_1^i+X_3^i)$ must be equal.

Write $d=d_2=2^{j}k$, with $j \geq 0$ and $k$ odd. The above argument yields that, in particular, 
$$L_d(X_1,X_3)= (X_1+X_3) (X_1^d+X_3^d)=(X_1+X_3)(X_1^k+X_3^k)^{2^j},$$
the highest homogeneous term in $(X_1+X_3)\sum_{0<i\leq d}b_{i,0}(X_1^i+X_3^i)$, is a square in $\overline{\mathbb{F}_q}[X_1,X_3]$.

\begin{itemize}
\item If $j>0$ then $(X_1+X_3)^{2^j+1}$ is the is the highest power of $(X_1+X_3)$ dividing $L_d(X_1,X_3).$ Since $2^j+1$ is odd, $L_d(X_1,X_3)$ is not a square in $\overline{\mathbb{F}_q}[X_1,X_3]$ and the surface $\mathcal{Y}$ is irreducible.
\item If $j=0$, following the same argument as above, if $\mathcal{Y}$ is reducible then $k=1$, a contradiction to $d>1.$
\end{itemize}

This shows that if $d=d_2>d_1>1$ then $\mathcal{Y}$ and therefore $\mathcal{S}_{f_1,f_2}$ is absolutely irreducible or contains an absolutely irreducible component defined over $\mathbb{F}_q$. This completes the proof.

\endproof

\begin{proposition}\label{Prop:p_even d_2+1<2d_1}
Let $\mathcal{S}_{f_1,f_2}$ be as in \eqref{Eq:S_f}, $d=d_2 > d_1>1$, $(j_1,j_2)=(1,0)$.  If $d_2+1 \leq 2d_1$ and $d_2$ is even then  $\mathcal{S}_{f_1,f_2}$ contains an absolutely irreducible component defined over $\mathbb{F}_q$.
\end{proposition}

\proof
Consider the curve 
$$\mathcal{C}: F(1,X_1,X_2,0,0)=X_2^2+X_2\sum_{i\leq d_1}a_i X_1^i+X_1\sum_{0<i\leq d}b_{i,0}X_1^i=0,$$
which is equivalent to 
$$\mathcal{C}^{\prime}: X_2^2+X_2+\frac{X_1\sum_{0<i\leq d}b_{i,0}X_1^i}{\left(\sum_{i\leq d_1}a_i X_1^i\right)^2}=0.$$

Such a  curve is absolutely irreducible if there exists $\eta\in \mathbb{P}^{1}(\overline{\mathbb{F}_q})=\overline{\mathbb{F}_q} \cup \{\infty\}$ such that 
$$2\nmid v_{\eta}\left(\frac{X_1\sum_{0<i\leq d}b_{i,0}X_1^i}{\left(\sum_{i\leq d_1}a_i X_1^i\right)^2}\right)<0.$$
Consider $\eta=\infty$. In this case 
$$2\nmid v_{\eta}\left(\frac{X_1\sum_{0<i\leq d}b_{i,0}X_1^i}{\left(\sum_{i\leq d_1}a_i X_1^i\right)^2}\right)=d_2+1-2d_1<0$$
since $d_2$ is even.

This shows that $\mathcal{C}^{\prime}$ and therefore $\mathcal{S}_{f_1,f_2}$ are both absolutely irreducible.
\endproof

The following theorem summarizes the results about the even characteristic case.
 \begin{theorem}\label{Th:finalp=2}
Let $\mathcal{S}_{f_1,f_2}$ be as in \eqref{Eq:S_f} contains an absolutely irreducible component defined over $\mathbb{F}_q$ unless $1<d_1<d_2\leq 2d_1-1$, $d_2$ odd, $j_2=0$.
\end{theorem}

\section{$p>2$}



\begin{proposition}\label{Proposition:6.1}
Let $\mathcal{S}_{f_1,f_2}$ be as in \eqref{Eq:S_f}, and $p>2$. If $d=d_2>d_1$ and $(j_1,j_2)=(1,1)$ then $\mathcal{S}_{f_1,f_2}$ contains an absolutely irreducible component defined over $\mathbb{F}_q$.
\end{proposition}
\proof
By Proposition \ref{d1d2}, we can assume that $b_{d-1,1} =0$. The surface $\mathcal{S}_{f_1,f_2}$ reads

\begin{eqnarray*}
(X_2-X_4)^2X_0^{d-1}+(X_2-X_4)\left[\sum_ia_i(X_1^i-X_3^i)X_0^{d-i}\right]&& \\ +(X_1-X_3)\left[\sum_{i>0}b_{i,0}(X_1^i-X_3^i)X_0^{d-i}+\sum_{0 \leq i \leq d-2}b_{i,1}(X_1^iX_2-X_3^iX_4)X_0^{d-i-1}\right]&=0&
\end{eqnarray*}

Let $\mathcal{W} = \mathcal{S}_{f_1,f_2} \cap (X_2=X_4)$ be defined by 

$$
(X_1-X_3) \left[\sum_{i>0}b_{i,0}(X_1^i-X_3^i)X_0^{d-i} +X_2 \sum_{0<i\leq d-2}b_{i,1}(X_1^i-X_3^i)X_0^{d-i-1}\right]=0.
$$

If $b_{i,1} \neq 0$ for some $0 < i \leq d-2$, then $\mathcal{W}$ is of degree 1 in $X_2$.

Thus $b_{i,1}=0$ for $i>0$ and Remark \ref{remark:b_01} yields the claim.

\endproof

\begin{proposition}\label{prop:d_2>d_1=1}
Let $\mathcal{S}_{f_1,f_2}$ be as in \eqref{Eq:S_f}, $d=d_2>d_1=1$, $(j_1,j_2)=(1,0)$, and $p>2$. If $d_2 \neq p^j$, with $j \geq 0$, or $b_{i,j} \neq 0$ for some $(i,j) \neq (d_2,0)$, or $h=\infty$ and $a_{d_1}=0$ then $\mathcal{S}_{f_1,f_2}$ contains an absolutely irreducible component defined over $\mathbb{F}_q$.
\end{proposition}

\proof
The surface $\mathcal{S}_{f_1,f_2}$ reads
\begin{eqnarray*}
(X_2-X_4)^2X_0^{d-1}+a_1(X_1-X_3)(X_2-X_4)X_0^{d-1} +(X_1-X_3)\left[\sum_{i\leq d_2}b_{i,0}(X_1^i-X_3^i)X_0^{d-i}\right].
\end{eqnarray*}

If $\mathcal{S}_{f_1,f_2}$ is reducible then by Remark \ref{remark:fattorizzazione}, there exist $r(X_1,X_3)$, $s(X_1,X_3)$ such that 
\begin{eqnarray*}
(X_1-X_3) \left[ \sum_{i \leq d_2}b_{i,0}(X_1^i-X_3^i)\right]&=&r(X_1,X_3)s(X_1,X_3),\\ a_1(X_1-X_3)&=&r(X_1,X_3)+s(X_1,X_3).
\end{eqnarray*}

Consider the affine part of $\mathcal{S}_{f_1,f_2}$ and note that the surface is reducible if and only if
$$\Delta(X_1,X_3)=a_1^2(X_1-X_3)^2-4(X_1-X_3)\sum_{i\leq d_2}b_{i,0}(X_1^i-X_3^i)$$
is a  square in $\overline{\mathbb{F}_q}[X_1,X_3]$.

Suppose that $\Delta(X_1,X_3)$ is a  square in $\overline{\mathbb{F}_q}[X_1,X_3]$. Then, the highest homogeneous term $L_{d_2}(X_1,X_3)=-4b_{d_2,0}(X_1-X_3)(X_1^{d_2}-X_3^{d_2})$  is  a square in $\overline{\mathbb{F}_q}[X_1,X_3]$ too. Write $d_2=kp^j$, with $p\nmid k$. Then $L_{d_2}(X_1,X_3)=-4b_{d_2,0}(X_1-X_3)\left(X_1^k-X_3^k\right)^{p^j}$. 

\begin{itemize}
    \item Let $k>1$. If $L_{d_2}(X_1,X_3)$ is a square in $\overline{\mathbb{F}_q}[X_1,X_3]$ then $R(X_1):=L_{d_2}(X_1,1)=(X_1-1)(X_1^k-1)^{p^j}$ is a square in $\overline{\mathbb{F}_q}[X_1]$. 
    
    Every root of $R(X_1)$ different from $1$ has algebraic multiplicity equal to $p^j$ and $R(X_1)$ cannot be a square in $\overline{\mathbb{F}_q}[X_1]$.
    
    \item Let now $k=1$ and $d_2=p^j$ for some $j\geq 0$. First, suppose $h<\infty$. Write 

\begin{eqnarray*}
\Delta(X_1,X_3)&=&L_{d_2}(X_1,X_3)+L_{h}(X_1,X_3)+\cdots\\
&=&\left((X_1-X_3)^{(p^j+1)/2}+M(X_1,X_3)+\cdots\right)^2,
\end{eqnarray*}
where $M(X_1,X_3)$ has degree $(h+1)-(p^j+1)/2$.  Now $L_{h}(X_1,X_3):=2(X_1-X_3)^{(p^j+1)/2}M(X_1,X_3)$ equals 
$$
\begin{cases}
-4b_{h,0}(X_1-X_3)(X_1^{h}-X_3^{h}),& \textrm{if } h>1;\\
(a_1^2-4b_{1,0})(X_1-X_3)^2,& \textrm{if } h=1.\\
\end{cases}$$

If $h>1$ then $(X_1-X_3)^{(p^j-1)/2}\mid (X_1^{h}-X_3^{h})$. Write $h=k_1p^{j_1}$, so that $(X_1^{h}-X_3^{h})=(X_1^{k_1}-X_3^{k_1})^{p^{j_1}}$. Since  $(X_1-X_3)^{p^{j_1}}$ is the highest power of $(X_1-X_3)$ dividing  $(X_1^{k_1}-X_3^{k_1})^{p^{j_1}}$, we must have $(p^j-1)/2\leq p^{j_1}$. Since $j_1\leq j-1$, 
$$p^j-1\leq 2p^{j_1}\leq 2p^{j}/p,$$
and therefore $p^j\leq \frac{p}{p-2}$. This yields $p^j=3$, that is $j=1$ and $p=3$. 
Now, $(X_1-X_3)^4+\alpha (X_1-X_3)(X_1^2-X_3^2)+\beta(X_1-X_3)^2$, $\alpha,\beta \in \mathbb{F}_q$, is a square in $\overline{\mathbb{F}_q}[X_1,X_3]$ only if $(X_1-X_3)^2+\alpha (X_1+X_3)+\beta$ is a square. By direct checking, this is possible only if $\alpha=\beta=0$.

If $h=1$, in this case $M(X_1,X_3)$ has degree $2-(p^j+1)/2$ and this yields $p^j=3$, a contradiction.

This shows that if there exists $i$, $0<i<d_2$, such that $b_{i,0} \neq 0$ then $\Delta(X_1,X_3)$ is not a square in $\overline{\mathbb{F}_q}[X_1,X_3]$.

If $h=\infty$, 
$r(X_1,X_3)s(X_1,X_3)=b_{d_2,0}(X_1-X_3)^{p^j+1}$ and $r(X_1,X_3)+s(X_1,X_3)=a_1 (X_1-X_3)$. Since $p^j+1>2$, $r(X_1,X_3)=\sqrt{-b_{d_2,0}} (X_1-X_3)^{(p^j+1)/2}=-s(X_1,X_3)$ and 
    $$a_1 (X_1-X_3)=r(X_1,X_3)+s(X_1,X_3)\equiv0,$$
  yielding $a_{1}= 0$.

\end{itemize}
\endproof

\begin{proposition}\label{prop:p>3}
Let $\mathcal{S}_{f_1,f_2}$ be as in \eqref{Eq:S_f}, and $p>3$. If $d=d_2>d_1>1$ and $(j_1,j_2)=(1,0)$ then $\mathcal{S}_{f_1,f_2}$ contains an absolutely irreducible component defined over $\mathbb{F}_q$, unless $(p,d_1,d_2,h)\in\{(5,2,5,3),(7,2,7,3),(7,3,7,5)\}$ or $d=d_2=2d_1-1\leq 11$. 
\end{proposition}
\proof

Note that  $\mathcal{S}_{f_1,f_2}$  reads
\begin{eqnarray*}
(X_2-X_4)^2X_0^{d-1}+(X_2-X_4)\left[\sum_{i\leq d_1}a_i(X_1^i-X_3^i)X_0^{d-i}\right] +(X_1-X_3)\left[\sum_{i\leq d_2}b_{i,0}(X_1^i-X_3^i)X_0^{d-i}\right].
\end{eqnarray*}

If $\mathcal{S}_{f_1,f_2}$ is reducible then by Remark \ref{remark:fattorizzazione}, there exist $r(X_1,X_3)$, $s(X_1,X_3)$ such that 
\begin{eqnarray*}
(X_1-X_3)\left[\sum_{i\leq d_2}b_{i,0}(X_1^i-X_3^i)\right]&=&r(X_1,X_3)s(X_1,X_3),\\ \left[\sum_{i\leq d_1}a_i(X_1^i-X_3^i)\right]&=&r(X_1,X_3)+s(X_1,X_3).
\end{eqnarray*}

If $\deg(r)> \deg(s)$ then we can suppose that   $\deg(r)=d_1$ and  $(X_1^{d_1}-X_3^{d_1})\mid (X_1-X_3)(X_1^{d_2}-X_3^{d_2})$. Since $d_1>1$, this yields $d_1\mid d_2$ and thus $d_1\leq d_2/2.$ 
  Now, 
    $$1+d_2=\deg(r)+\deg(s) \leq  d_1+d_1-1\leq d_2-1,$$
    a contradiction.

From now on we consider the case  $\deg(r)=\deg(s)=(d_2+1)/2$. Then $d_1 \leq (d_2+1)/2$. Recall that $\mathcal{S}_{f_1,f_2}$ is reducible if and only if
$$
\Delta(X_1,X_3):=\left[ \sum_i a_i (X_1^i-X_3^i) \right]^2-4(X_1-X_3) \left[ \sum_i b_{i,0}(X_1^i-X_3^i) \right]
$$
is a square in $\overline{\mathbb{F}_q}[X_1,X_3]$. We distinguish several cases.

\begin{itemize}

\item[1)] If $d_1 \leq d_2/2$ the homogenous part of the highest degree in $\Delta(X_1,X_3)$ is $L(X_1,X_3):=-4b_{d_2,0}(X_1-X_3)(X_1^{d_2}-X_3^{d_2})$. If $\Delta(X_1,X_3)$ is a square in $\overline{\mathbb{F}_q}[X_1,X_3]$ so it is $L(X_1,X_3)$. Let $d_2=kp^j$ for some non-negative $j$ and $p \nmid k$.

If $k>1$ then  $L(X_1,X_3)$ cannot be a square in $\overline{\mathbb{F}_q}[X_1,X_3]$.
    
From now on we can assume that  $k=1$ and then $d_2=p^j$. In this case $d_1\leq (d_2-1)/2$.

     Suppose  that $b_{i,0}=0$, for any $0<i<d_2$, i.e. $h=\infty$. 
    In this case $r(X_1,X_3)s(X_1,X_3)=b_{d_2,0}(X_1-X_3)^{p^j+1}$ and $r(X_1,X_3)+s(X_1,X_3)=\sum_i a_i (X_1^i-X_3^i)$. Since $p^j+1>2d_1$, $r(X_1,X_3)=\sqrt{-b_{d_2,0}} (X_1-X_3)^{(p^j+1)/2}=-s(X_1,X_3)$ and 
    $$\sum_i a_i (X_1^i-X_3^i)=r(X_1,X_3)+s(X_1,X_3)\equiv0,$$
   a contradiction to $a_{d_1}\neq 0$.
     
        

    From now on we suppose that  $h<\infty$. We distinguish three cases:
    \begin{itemize}
        \item[a)] $h+1>2d_1$. Write
        $$
        \Delta(X_1,X_3)=\left( (X_1-X_3)^{(p^j+1)/2} + M(X_1,X_3) + \ldots \right)^2,
        $$
        where $M(X_1,X_3)$ has degree $(h+1)-(p^j+1)/2$. Now
        $$
        2(X_1-X_3)^{(p^j+1)/2}M(X_1,X_3)=-4b_{h,0}(X_1-X_3)(X_1^h-X_3^h)
        $$
        and then $(X_1-X_3)^{(p^j-1)/2} \mid (X_1^h-X_3^h)$. Write $h=k_1p^{j_1}$, so that $(X_1^h-X_3^h)=(X_1^{k_1}-X_3^{k_1})^{p^{j_1}}$. Since $(X_1-X_3)^{p^{j_1}}$ is the highest power of $(X_1-X_3)$ dividing $(X_1^{k_1}-X_3^{k_1})^{p^{j_1}}$, we must have $(p^j-1)/2 \leq p^{j_1}$. Since $j_1 \leq j-1$,
        $$
        p^j-1 \leq 2p^{j_1} \leq 2p^j/p,
        $$
        and therefore $p^j \leq \frac{p}{p-2}$. This yields $p^j=3$, a contradiction to $p>3$.

         \item[b)] $h+1<2d_1$. Write
        $$
        \Delta(X_1,X_3)=\left( (X_1-X_3)^{(p^j+1)/2} + M(X_1,X_3) + \ldots \right)^2,
        $$
        where $M(X_1,X_3)$ has degree $2d_1-(p^j+1)/2$. Now
        $$
        2(X_1-X_3)^{(p^j+1)/2}M(X_1,X_3)=a_{d_1}^2(X_1^{d_1}-X_3^{d_1})^2
        $$
        and then $(X_1-X_3)^{(p^j+1)/2} \mid (X_1^{d_1}-X_3^{d_1})^2$. Write $d_1=k_1p^{j_1}$, so that $(X_1^{d_1}-X_3^{d_1})^2=(X_1^{k_1}-X_3^{k_1})^{2p^{j_1}}$.  We must have $(p^j+1)/2 \leq 2p^{j_1}$. Since $j_1 \leq j-1$,
        $$
        p^j+1 \leq 4p^{j_1} \leq 4p^j/p,
        $$
        and therefore $(p-4)p^j \leq -p$, a contradiction to $p>3$. 
        
        
        \item[c)] $h+1=2d_1$. Write
        \begin{eqnarray*}
        r(X_1,X_3)&=&r_{(d_2+1)/2}(X_1,X_3)+\cdots+r_{d_1}(X_1,X_3)+\cdots \\
        s(X_1,X_3)&=&s_{(d_2+1)/2}(X_1,X_3)+\cdots+s_{d_1}(X_1,X_3)+\cdots,
        \end{eqnarray*}
        where $s_i,r_i$ homogeneous of degree $i$ or the zero polynomial. 
        Since $r(X_1,X_3)+s(X_1,X_3)=\left[\sum_{i\leq d_1}a_i(X_1^i-X_3^i)\right]$, $s_{(d_2+1)/2}(X_1,X_3)=-r_{(d_2+1)/2}(X_1,X_3)$, \ldots, $s_{d_1+1}(X_1,X_3)=- r_{d_1+1}(X_1,X_3)$.
        
        The term of degree $d_2>2d_1$ in $r(X_1,X_3)s(X_1,X_3)$ is 
        $$r_{(d_2+1)/2}s_{(d_2-1)/2}+r_{(d_2-1)/2}s_{(d_2+1)/2}=-2r_{(d_2+1)/2}r_{(d_2-1)/2}\equiv 0,$$
        and then $s_{(d_2-1)/2}=r_{(d_2-1)/2}\equiv 0$.

        By induction, the term of degree $i>(d_2+1)/2+d_1\geq 2d_1=h+1$ in $r(X_1,X_3)s(X_1,X_3)$ is 
        $$r_{(d_2+1)/2}s_{i-(d_2+1)/2}+r_{i-(d_2+1)/2}s_{(d_2+1)/2}=-2r_{(d_2+1)/2}r_{i-(d_2+1)/2}\equiv 0,$$
        and then $s_{i-(d_2+1)/2}=r_{i-(d_2+1)/2}\equiv 0$. 
        
        Consider $i=(d_2+1)/2+d_1>2d_1=h+1$ in $r(X_1,X_3)s(X_1,X_3)$ is 
        \begin{eqnarray*}
        r_{(d_2+1)/2}s_{d_1}+r_{d_1}s_{(d_2+1)/2}&=&r_{(d_2+1)/2}(s_{d_1}-r_{d_1})\\
        &=&\sqrt{-b_{d_2,0}}(X_1-X_3)^{(d_2+1)/2}(s_{d_1}-r_{d_1}),\\
        &\equiv&0
        \end{eqnarray*}
        Then $r_{d_1}=s_{d_1}=a_{d_1}(X_1^{d_1}-X_3^{d_1})/2$.
        

        Suppose that $2d_1<(d_2+1)/2$. The term of degree $2d_1=h+1$ in  $r(X_1,X_3)s(X_1,X_3)$ is 
        \begin{eqnarray*}
        &&r_{d_1}s_{d_1}=b_{h,0}(X_1-X_3)(X_1^h-X_3^h)=a_{d_1}^2 (X_1^{d_1}-X_3^{d_1})^2/4.
        \end{eqnarray*}
        
        If $p\nmid d_1$ and $p\mid h$, then $(X_1-X_3)^2\mid \mid (X_1^{d_1}-X_3^{d_1})^2$ but $(X_1-X_3)^{p+1}\mid (X_1-X_3)(X_1^h-X_3^h)$.

        If $p\nmid d_1$ and $p\nmid h$, then $(X_1^h-X_3^h)/(X_1-X_3)$ is separable, whereas $(X_1^{d_1}-X_3^{d_1})^2/(X_1-X_3)^2$ is not.

        If $p\mid d_1$ then $p\nmid h$ and so $(X_1-X_3)^{2p} \mid (X_1^{d_1}-X_3^{d_1})^2$ and  $(X_1-X_3)^2\mid  \mid (X_1-X_3)(X_1^h-X_3^h)$. In all  the cases a contradiction arises.

        Suppose that $2d_1\geq (d_2+1)/2$.
        The term of degree $2d_1=h+1$ in  $r(X_1,X_3)s(X_1,X_3)$ is 
        \begin{eqnarray*}
        &&\!\!\!\!\!\!\!\!\!\!\!\!r_{d_1}s_{d_1}+r_{(d_2+1)/2}s_{2d_1-(d_2+1)/2}+r_{2d_1-(d_2+1)/2}s_{(d_2+1)/2}=b_{h,0}(X_1-X_3)(X_1^h-X_3^h)\\
        &&\!\!\!\!\!\!\!\!\!\!\!\!=a_{d_1}^2 (X_1^{d_1}-X_3^{d_1})^2/4+\sqrt{-b_{d_2,0}}(X_1-X_3)^{(d_2+1)/2}(s_{2d_1-(d_2+1)/2}-r_{2d_1-(d_2+1)/2})
        \end{eqnarray*}
        
        If $p\nmid d_1$ and $p\mid h$, then $(X_1-X_3)^2\mid \mid (X_1^{d_1}-X_3^{d_1})^2$ but $(X_1-X_3)^{p+1}\mid (X_1-X_3)(X_1^h-X_3^h)$ and $(X_1-X_3)^{(p+1)/2}\mid (X_1-X_3)^{(d_2+1)/2}$.
        
        If $p\nmid d_1$ and $p\nmid h$, then $(X_1-X_3)^2\mid \mid (X_1^{d_1}-X_3^{d_1})^2$ and $(X_1-X_3)^{2}\mid\mid (X_1-X_3)(X_1^h-X_3^h)$ and $(X_1-X_3)^{5}\mid (X_1-X_3)^{(d_2+1)/2}$, unless $d_2=p^j\leq  7$.
        So, suppose that $d_2>7$.  If $s_{2d_1-(d_2+1)/2}=r_{2d_1-(d_2+1)/2}$ the same argument as for the case  $2d_1< (d_2+1)/2$ applies. 
        
        If $s_{2d_1-(d_2+1)/2}\neq r_{2d_1-(d_2+1)/2}$, we have 
        $$(X_1-X_3)^{(d_2-3)/2} \Big | \frac{b_{h,0}(X_1^h-X_3^h)}{X_1-X_3}-\frac{a_{d_1}^2(X_1^{d_1}-X_3^{d_1})^2}{4(X_1-X_3)^2}.$$
        This is equivalent to 
    \begin{eqnarray*}
    X_1^{(d_2-3)/2}\!\! \!\!\!\! &\Big |&\!\!\!\!\!\! H(X_1,X_3) :=  \frac{b_{h,0}((X_1+X_3)^h-X_3^h)}{X_1}-\frac{a_{d_1}^2((X_1+X_3)^{d_1}-X_3^{d_1})^2}{4X_1^2}\\
    &=&\!\!\!(b_{h,0}h-a_{d_1}^2 d_1^2/4)X_3^{h-1}+\left(b_{h,0}\binom{h}{2}-\frac{a_{d_1}^2}{2} d_1 \binom{d_1}{2}\right)  X_1X_3^{h-2}+\\
    &&+\left(b_{h,0}\binom{h}{3}-\frac{a_{d_1}^2}{2}d_1\binom{d_1}{3}-\frac{a_{d_1}^2}{4}\binom{d_1}{2}^2\right)X_1^2X_3^{h-3}+\cdots
    \end{eqnarray*}
        which yields 
        $b_{h,0}h-a_{d_1}^2d_1^2/4 =0$. 
        Since we are supposing that $d_2>7$, so that $(d_2-3)/2>2$. If $h\not \equiv 1 \pmod{p}$ then   $d_1\not \equiv 1 \pmod{p}$ and  
        $$b_{h,0}\binom{h}{3}-\frac{a_{d_1}^2}{2}d_1\binom{d_1}{3}-\frac{a_{d_1}^2}{4}\binom{d_1}{2}^2\equiv 0 \pmod{p},$$
        that is 
        $$\frac{a_{d_1}^2 d_1^2 (h-1)(h-2)}{24}-\frac{a_{d_1}^2 d_1^2(d_1-1)(d_1-2)}{12}-\frac{a_{d_1}^2d_1^2(d_1-1)^2}{16}\equiv 0 \pmod{p}.$$
        Since $h=2d_1-1$, one gets $d_1\equiv 1\pmod p$, a contradiction.

        If $h\equiv 1 \pmod{p}$ then $d_1\equiv 1 \pmod{p}$. Consider $p^{i_1}\mid \mid (h-1)$ and $p^{i_2}\mid \mid (d_1-1)$, since $d_1-1=(h-1)/2$, $i_1=i_2$. Now,
        \begin{eqnarray*}
    H(X_1,X_3) &=& (b_{h,0}h-a_{d_1}^2 d_1^2/4)X_3^{h-1}+\\ &&\left(b_{h,0}\binom{h}{p^{i_1}}-\frac{a_{d_1}^2}{2} d_1\binom{d_1}{p^{i_1}}\right)X_1^{p^{i_1}-1}X_3^{h-p^{i_1}}+\\
     &&+\left(b_{h,0}\binom{h}{p^{i_1}+1}-\frac{a_{d_1}^2}{2} d_1\binom{d_1}{p^{i_1}+1}\right)X_1^{p^{i_1}}X_3^{h-p^{i_1}-1}\\
     &&+\left(b_{h,0}\binom{h}{2p^{i_1}-1}-\frac{a_{d_1}^2}{4}\binom{d_1}{p^{i_1}}^2\right)X_1^{2p^{i_1}-2}X_3^{h-2p^{i_1}+1}+\cdots.
    \end{eqnarray*}
    Since $b_{h,0}h-a_{d_1}^2 d_1^2/4=0$, 
    $$H(X_1,X_3)= \left(b_{h,0}\binom{h}{2p^{i_1}-1}-\frac{a_{d_1}^2}{4}\binom{d_1}{p^{i_1}}^2\right)X_1^{2p^{i_1}-2}X_3^{h-2p^{i_1}+1}+\cdots$$
    and therefore $(p^j-3)/2\leq 2p^{i_1}-2$, a contradiction to $p>3$.
        
        If $p\mid d_1$ then $p\nmid h$ and so $(X_1-X_3)^2\mid \mid (X_1-X_3)(X_1^{h}-X_3^{h})$ and  $(X_1-X_3)^{2p} \mid (X_1^{d_1}-X_3^{d_1})^2$ and $(X_1-X_3)^3 \mid (X_1-X_3)^{(d_2+1)/2}$, since $d_2\geq p>3$.

    \end{itemize}

\item [2)] If $d_2+1=2d_1$ then the homogeneous part of the highest degree $L(X_1,X_3)$ in $\Delta(X_1,X_3)$ is $$a_{d_1}^2(X_1^{d_1}-X_3^{d_1})^2-4b_{2d_1-1,0}(X_1-X_3)(X_1^{2d_1-1}-X_3^{2d_1-1}).$$
In the following we will prove that, if $d_1>6$ , then 

$$L(X_1,1)=(a_{d_1}^2-4b_{2d_1-1,0})X_1^{2d_1}+4b_{2d_1-1,0}X_1^{2d_1-1}-2a_{d_1}^2X_1^{d_1}+4b_{2d_1-1,0} X_1+(a_{d_1}^2-4b_{2d_1-1,0})$$
is not a square in $\overline{\mathbb{F}_q}[X_1]$. 
 
First, if $a_{d_1}^2-4b_{2d_1-1,0}=0$ then $X_1\mid \mid L(X_1,1)$ and therefore $L(X_1,1)$ is not a square in $\overline{\mathbb{F}_q}[X_1]$. So we can suppose $a_{d_1}^2-4b_{2d_1-1,0}\neq0$.

In the following we will consider the  derivatives of $L(X_1,1)$ 

\begin{eqnarray*}
L^{(1)}(X_1,1) &:=& 2d_1(a_{d_1}^2-4b_{2d_1-1,0})X_1^{2d_1-1}+4(2d_1-1)b_{2d_1-1,0}X_1^{2d_1-2}\\
&&-2d_1a_{d_1}^2X_1^{d_1-1}+4b_{2d_1-1,0},\\
L^{(2)}(X_1,1) &:=& 2d_1(2d_1-1)(a_{d_1}^2-4b_{2d_1-1,0})X_1^{2d_1-2}\\
&&+4(2d_1-1)(2d_1-2)b_{2d_1-1,0}X_1^{2d_1-3}-2d_1(d_1-1)a_{d_1}^2X_1^{d_1-2},\\
L^{(3)}(X_1,1) &:=& 2d_1(2d_1-1)(2d_1-2)(a_{d_1}^2-4b_{2d_1-1,0})X_1^{2d_1-3}\\
&&+4(2d_1-1)(2d_1-2)(2d_1-3)b_{2d_1-1,0}X_1^{2d_1-4}\\
&&-2d_1(d_1-1)(d_1-2)a_{d_1}^2X_1^{d_1-3},\\
L^{(4)}(X_1,1) &:=& 2d_1(2d_1-1)(2d_1-2)(2d_1-3)(a_{d_1}^2-4b_{2d_1-1,0})X_1^{2d_1-4}\\
&&+4(2d_1-1)(2d_1-2)(2d_1-3)(2d_1-4)b_{2d_1-1,0}X_1^{2d_1-5}\\
&&-2d_1(d_1-1)(d_1-2)(d_1-3)a_{d_1}^2X_1^{d_1-4}.\\
\end{eqnarray*}

Clearly $X_1=0$ cannot be a repeat root of $L(X_1,1)$.

$\bullet$ If $d_1\equiv 0\pmod p$, then $L^{(1)}(X_1,1)=-4b_{2d_1-1,0}(X_1^{2d_1-2}-1)$. Now,  $X_1^{d_1-1}=1$  or $X_1^{d_1-1}=-1$ yield, together with $L(X_1,1)=0$, $X_1=1$ or $(a_{d_1}^2-4b_{2d_1-1,0})X_1^{2}+(8b_{2d_1-1,0}+2a_{d_1}^2) X_1+(a_{d_1}^2-4b_{2d_1-1,0})=0$. Since $L^{(2)}(X_1,1)=8b_{2d_1-1,0}X_1^{2d_1-3}$, all the roots have at most multiplicity $2$. So, $L(X_1,1)$ cannot be a square in $\overline{\mathbb{F}_q}[X_1]$ since its degree is larger than $6$.

$\bullet$ If $d_1\equiv 1\pmod p$, since $$X_1L^{\prime}(X_1,1)-L(X_1,1)=(a_{d_1}^2-4b_{2d_1-1,0})(X_1^{2d_1}-1),$$
the only repeated roots of $L(X_1,1)$ satisfy $X_1^{2d_1}=1$. If $X_1^{d_1}=1$ or $X_1^{d_1}=-1$, then $L(X_1,1)=0$ yields $X_1=1$ or $4b_{2d_1-1,0}X_1^2+4(a_{d_1}^2-b_{2d_1-1,0})X_1+4b_{2d_1-1,0}=0$, respectively. This shows that the repeated roots of $L(X_1,1)$ are at most three. Since in this case $L^{(2)}(X_1,1)$ reads $2d_1(2d_1-1)(a_{d_1}^2-4b_{2d_1-1,0})X_1^{2d_1-2}$, there are not roots of $L(X_1,1)$ of multiplicity larger than $2$. Our assumption $d_1>6$ is sufficient to prove that   $L(X_1,1)$ cannot be a square in $\overline{\mathbb{F}_q}[X_1]$, since its degree is larger than $ 6$.

$\bullet$ If $2d_1\equiv 1\pmod p$, then 
\begin{eqnarray*}
L^{(1)}(X_1,1)&=& (a_{d_1}^2-4b_{2d_1-1,0})X_1^{2d_1-1}-a_{d_1}^2X_1^{d_1-1}+4b_{2d_1-1,0}\\
L(X_1,1)-(X_1-1)L^{(1)}(X_1,1)&=& a_{d_1}^2(X_1^{d_1}-1)(X_1^{d_1-1}-1).
\end{eqnarray*}
Both $X_1^{d_1}=1$ and $X_1^{d_1-1}=1$ combined with $L(X_1,1)=0$ give $X_1=1$. So, the only repeated root of $L(X_1,1)$ is $X_1=1$. Since $L^{(2)}(X_1,1)= -(d_1-1)a_{d_1}^2X_1^{d_1-2}$, $X_1=1$ is a double root of $L(X_1,1)$, which proves that   $L(X_1,1)$ cannot be a square in $\overline{\mathbb{F}_q}[X_1]$ (its degree is larger than $ 12$).
 
 From now on we can suppose that $d_1(d_1-1)(2d_1-1)\not\equiv 0 \pmod p$. 
First, note that $X_1=1$ is a root of multiplicity at most $4$ for $L(X_1,1)$. In fact,
\begin{eqnarray*}
L(X_1+1,1)\!\!\!\!&=&\!\!\!\!(a_{d_1}^2-4b_{2d_1-1,0})(X_1+1)^{2d_1}+4b_{2d_1-1,0}(X_1+1)^{2d_1-1}\\
&&-2a_{d_1}^2(X_1+1)^{d_1}+4b_{2d_1-1,0} X_1+a_{d_1}^2\\
\!\!\!\!&=&\!\!\!\!\left(a_{d_1}^2\left(\binom{2d_1}{2}-2\binom{d_1}{2}\right) -4b_{2d_1-1,0}\left( \binom{2d_1}{2}-\binom{2d_1-1}{2}\right)\right)X_1^2+\\
&&\!\!\!\!\left(a_{d_1}^2\left(\binom{2d_1}{3}-2\binom{d_1}{3}\right) -4b_{2d_1-1,0}\left( \binom{2d_1}{3}-\binom{2d_1-1}{3}\right)\right)X_1^3+\\
&&\!\!\!\!\left(a_{d_1}^2\left(\binom{2d_1}{4}-2\binom{d_1}{4}\right) -4b_{2d_1-1,0}\left( \binom{2d_1}{4}-\binom{2d_1-1}{4}\right)\right)X_1^4+\cdots\\
\!\!\!\!&=&\!\!\!\!\left(a_{d_1}^2d_1^2 -4b_{2d_1-1,0}(2d_1-1)\right)X_1^2+\\
&&\!\!\!\!\left(a_{d_1}^2d_1^2(d_1-1) -4b_{2d_1-1,0}(2d_1-1)(d_1-1)\right)X_1^3+\\
&&\!\!\!\!\frac{\left(a_{d_1}^2\left(7d_1^4 - 17d_1^3 + 8d_1^2 + 2d_1\right) -4b_{2d_1-1,0}\left(16 d_1^3 - 48 d_1^2 + 44 d_1 - 12\right)\right)}{12}X_1^4+\cdots\\
\end{eqnarray*} 
If both the coefficients of $X_1^2$ and $X_1^4$ vanish then $d_1^2(d_1-1)(2d_1-1)a_{d_1}^2=0$ a contradiction.

 Denote  $Y=X_1^{d_1}$, so that 
\begin{eqnarray*}
L(X_1,1)&=&(a_{d_1}^2-4b_{2d_1-1,0})Y^2+4b_{2d_1-1,0}\frac{Y^2}{X_1}-2a_{d_1}^2Y\\
&&+4b_{2d_1-1,0} X_1+(a_{d_1}^2-4b_{2d_1-1,0})\\
L^{(1)}(X_1,1) &=& 2d_1(a_{d_1}^2-4b_{2d_1-1,0})\frac{Y^2}{X_1}+4(2d_1-1)b_{2d_1-1,0}\frac{Y^2}{X_1^2}\\
&&-2d_1a_{d_1}^2\frac{Y}{X_1}+4b_{2d_1-1,0},\\
L^{(2)}(X_1,1) &=& 2d_1(2d_1-1)(a_{d_1}^2-4b_{2d_1-1,0})\frac{Y^2}{X_1^2}+4(2d_1-1)(2d_1-2)b_{2d_1-1,0}\frac{Y^2}{X_1^3}\\
&&-2d_1(d_1-1)a_{d_1}^2\frac{Y}{X_1^2},\\
L^{(3)}(X_1,1) &=& 2d_1(2d_1-1)(2d_1-2)(a_{d_1}^2-4b_{2d_1-1,0})\frac{Y^2}{X_1^3}\\
&&+4(2d_1-1)(2d_1-2)(2d_1-3)b_{2d_1-1,0}\frac{Y^2}{X_1^4}\\
&&-2d_1(d_1-1)(d_1-2)a_{d_1}^2\frac{Y}{X_1^3},\\
L^{(4)}(X_1,1) &=& 2d_1(2d_1-1)(2d_1-2)(2d_1-3)(a_{d_1}^2-4b_{2d_1-1,0})\frac{Y^2}{X_1^4}\\
&&+4(2d_1-1)(2d_1-2)(2d_1-3)(2d_1-4)b_{2d_1-1,0}\frac{Y^2}{X_1^5}\\
&&-2d_1(d_1-1)(d_1-2)(d_1-3)a_{d_1}^2\frac{Y}{X_1^4}.
\end{eqnarray*}

The resultant of $L(X_1,1)$ and $L^{(1)}(X_1,1)$ with respect to $Y$ reads
$$
16b_{2d_1-1,0} (a_{d_1}^2-4b_{2d_1-1,0})X_1^2(X_1-1)^2Z(X_1),
$$
where 
\begin{eqnarray*}
Z(X_1)&:=&b_{2d_1-1,0}(2 d_1 - 1)^2(a_{d_1}^2 - 4 b_{2d_1-1,0})X_1^2
+(a_{d_1}^4 d_1^2 - 8 a_{d_1}^2 b_{2d_1-1,0} d_1^2\\ && + 2 a_{d_1}^2 b_{2d_1-1,0}+ 32 b_{2d_1-1,0}^2 d_1^2 - 32 b_{2d_1-1,0}^2 d_1 + 8 b_{2d_1-1,0}^2)X_1\\
&&+b_{2d_1-1,0}(2 d_1 - 1)^2(a_{d_1}^2 - 4 b_{2d_1-1,0}).
\end{eqnarray*}

Multiple roots of $L(X_1,1)$ distinct from $X_1=1$ must satisfy $Z(X_1)=0$.

$\bullet$ If $d_1\equiv 2\pmod p$, then $L^{(4)}(X_1,1)=24(a_{d_1}^2-4b_{2d_1-1,0})X_1^{2d_1-4}$ and all the roots have multiplicity at most $4$. This shows that   $L(X_1,1)$ cannot be a square in $\overline{\mathbb{F}_q}[X_1]$ (its degree is larger than $ 12$).

$\bullet$ If $d_1\not\equiv 2\pmod p$, then the resultant of $L^{(2)}(X_1,1)/Y$ and $L^{(3)}(X_1,1)/Y$ with respect to $Y$ reads
$$
4(2 d_1 - 1) a_{d_1}^2 d_1 (d_1-1) X_1(X_1 a_{d_1}^2 d_1^2 - 4 X_1 b_{2d_1-1,0} d_1^2 + 4 b_{2d_1-1,0} d_1^2 - 8 b_{2d_1-1,0} d_1 + 4 b_{2d_1-1,0}),
$$
where as the resultant of $L^{(3)}(X_1,1)/Y$ and $L^{(4)}(X_1,1)/Y$ with respect to $Y$ reads

\begin{eqnarray*}
8(d_1-2)(d_1 - 1)^2 d_1a_{d_1}^2(2 d_1-1)X_1\cdot\\
\cdot(X_1 a_{d_1}^2 d_1^2 - 4 X_1 b_{2d_1-1,0} d_1^2 + 4 b_{2d_1-1,0} d_1^2 - 10 b_{2d_1-1,0} d_1 + 6 b_{2d_1-1,0}).
\end{eqnarray*}
The unique possible root of multiplicity larger than $4$ must satisfy 
$$X_1= -\frac{4 b_{2d_1-1,0} (d_1-1)^2}{d_1^2(a_{d_1}^2-4b_{2d_1-1,0})}= -\frac{2b_{2d_1-1,0}( 2 d_1^2 - 5 d_1 + 3) }{d_1^2(a_{d_1}^2-4b_{2d_1-1,0})}$$
which implies $d_1(d_1 - 1)b_{2d_1-1,0}(a_{d_1}^2 - 4 b_{2d_1-1,0})=0$, a contradiction. Therefore, all the (at most three) multiple roots of $L(X_1,1)$ have multiplicity at most $4$ and so $L(X_1,1)$ cannot be a square in $\overline{\mathbb{F}_q}[X_1]$ (its degree is larger than $ 12$). 

\endproof

\begin{proposition}\label{prop:p=3}
Let $\mathcal{S}_{f_1,f_2}$ be as in \eqref{Eq:S_f}, and $p=3$. If $d=d_2>d_1>1$ and $(j_1,j_2)=(1,0)$ then $\mathcal{S}_{f_1,f_2}$ contains an absolutely irreducible component defined over $\mathbb{F}_q$, unless one of the following holds
\begin{enumerate}
    \item $d=d_2=2d_1-1\leq 11$;
    \item $h+1<2d_1< d_2=3^j$ and $3 \mid d_1$;
    \item $h+1=2d_1<d_2=3^j$ and $3 \mid (d_1-1)$.
\end{enumerate} 
\end{proposition}

\proof
The proof is similar to that of Proposition \ref{prop:p>3} and we use the same notations. 




If $d_2+1=2d_1$ the claim follows as in Proposition \ref{prop:p>3}. From now on we assume $d_1 \leq d_2/2$. The homogenous part of the highest degree in $\Delta(X_1,X_3)$ is $L(X_1,X_3):=-4b_{d_2,0}(X_1-X_3)(X_1^{d_2}-X_3^{d_2})$. If $\Delta(X_1,X_3)$ is a square in $\overline{\mathbb{F}_q}[X_1,X_3]$ so it is $L(X_1,X_3)$. Let $d_2=k \cdot 3^j$ for some non-negative $j$ and $3 \nmid k$.

If $k>1$ then  $L(X_1,X_3)$ cannot be a square in $\overline{\mathbb{F}_q}[X_1,X_3]$.
    
From now on we can assume that  $k=1$ and then $d_2=3^j$. In this case $d_1\leq (d_2-1)/2$ and $j>1$.

If $h=\infty$ the claim follows as in Proposition  \ref{prop:p>3}. 

In the case $h<\infty$ we distinguish three cases:
    \begin{itemize}
         \item[a)] $h+1>2d_1$. Following the same argument as in the proof of Proposition \ref{prop:p>3}, we get $d=d_2=3$ and therefore $1<d_1\leq 1$.

        \item[b)] $h+1<2d_1$. Write
        \begin{eqnarray*}
        r(X_1,X_3)&=&r_{(d_2+1)/2}(X_1,X_3)+\cdots+r_{d_1}(X_1,X_3)+\cdots \\
        s(X_1,X_3)&=&s_{(d_2+1)/2}(X_1,X_3)+\cdots+s_{d_1}(X_1,X_3)+\cdots,
        \end{eqnarray*}
        where $s_i,r_i$ homogeneous of degree $i$ or the zero polynomial. As in Proposition \ref{prop:p>3} it follows that $s_{(d_2+1)/2}=-r_{(d_2+1)/2}=-\sqrt{-b_{d_2,0}}(X_1-X_3)^{(d_2+1)/2}$,
$s_{i}(X_1,X_3)=r_{i}(X_1,X_3) \equiv 0$, for any $d_1 < i <  (d_2+1)/2$, and
$r_{d_1}(X_1,X_3)=s_{d_1}(X_1,X_3)=a_{d_1}(X_1^{d_1}-X_3^{d_1})/2$. 

Suppose that $2d_1<(d_2+1)/2$. The term of degree $2d_1>h+1$ in  $r(X_1,X_3)s(X_1,X_3)$ is 
        \begin{eqnarray*}
        &&r_{d_1}s_{d_1}=a_{d_1}^2 (X_1^{d_1}-X_3^{d_1})^2/4 \equiv 0,
        \end{eqnarray*}
        a contradiction to $a_{d_1} \neq 0$.
        
        Suppose that $2d_1\geq (d_2+1)/2$.
        The term of degree $2d_1>h+1$ in  $r(X_1,X_3)s(X_1,X_3)$ is 
        \begin{eqnarray*}
        r_{d_1}s_{d_1}+r_{(d_2+1)/2}s_{2d_1-(d_2+1)/2}+r_{2d_1-(d_2+1)/2}s_{(d_2+1)/2}&=&\\
        a_{d_1}^2 (X_1^{d_1}-X_3^{d_1})^2/4+\sqrt{-b_{d_2,0}}(X_1-X_3)^{(d_2+1)/2}(s_{2d_1-(d_2+1)/2}-r_{2d_1-(d_2+1)/2})&\equiv&0
        \end{eqnarray*}
        Then
        $$
        a_{d_1}^2 (X_1^{d_1}-X_3^{d_1})^2/4=\sqrt{-b_{d_2,0}}(X_1-X_3)^{(d_2+1)/2}(r_{2d_1-(d_2+1)/2}-s_{2d_1-(d_2+1)/2}).
        $$
        If $s_{2d_1-(d_2+1)/2}=r_{2d_1-(d_2+1)/2}$ then $a_{d_1}$ must be zero.\\
        If $s_{2d_1-(d_2+1)/2} \neq r_{2d_1-(d_2+1)/2}$ and $3 \nmid d_1$, then $(X_1-X_3)^2 || (X_1^{d_1}-X_3^{d_1})^2$ but $(d_2+1)/2>2$. In both cases a contradiction arises.

        \item[c)] $h+1=2d_1$. If $d_1 \not\equiv 1 \pmod 3$, the claim follows as Proposition \ref{prop:p>3}.
        \end{itemize}
       
\endproof

\end{itemize}

We are now in position to prove the main result of this paper.

\noindent \emph{ Proof of the Main Theorem.}
By Remark \ref{remark:4.1}, Theorem \ref{Th:finalp=2}, Propositions \ref{Proposition:6.1}, \ref{prop:d_2>d_1=1}, \ref{prop:p>3}, \ref{prop:p=3} if  there exists $b_{i,j}\neq0$ with  $j\neq 0$ and none of the conditions listed above holds, then $\mathcal{S}_{f_1,f_2}$ defined as in \eqref{Eq:S_f} contains an absolutely irreducible $\mathbb{F}_q$-rational component. By Theorem \ref{Th:Main}  $\mO_5(f_1,f_2)$ is not an ovoid. 





\end{document}